\documentclass[12pt]{amsart}
 \usepackage{amscd,amsmath,amsthm,amssymb}
 \usepackage{pstcol,pst-plot,pst-3d}
 \usepackage{lineno}
 \psset{unit=0.7cm,linewidth=0.8pt,arrowsize=2.5pt 4}

\newpsstyle{fatline}{linewidth=1.5pt}
\newpsstyle{fyp}{fillstyle=solid,fillcolor=verylight}
\definecolor{verylight}{gray}{0.97}
\definecolor{light}{gray}{0.93}
\definecolor{medium}{gray}{0.82}
 \def\NZQ{\Bbb}               
 
 \def\QQ{{\NZQ Q}}
 \def\ZZ{{\NZQ Z}}

 \def\frk{\frak}               

 \def\mm{{\frk m}}
 
 \def\B{{\mathcal B}}
 \def\P{{\mathcal P}}

 \def\ab{{\bold a}}
 \def\bb{{\bold b}}
 \def\xb{{\bold x}}

 \def\cb{{\bold c}}

 \def\opn#1#2{\def#1{\operatorname{#2}}} 
 \opn\chara{char} \opn\length{\ell} \opn\pd{pd} \opn\rk{rk}
 \opn\projdim{proj\,dim} \opn\injdim{inj\,dim} \opn\rank{rank}
 \opn\depth{depth} \opn\grade{grade} \opn\height{height}
 \opn\embdim{emb\,dim} \opn\codim{codim}
 
 \opn\Tr{Tr} \opn\bigrank{big\,rank}
 \opn\superheight{superheight}\opn\lcm{lcm}
 \opn\trdeg{tr\,deg}
 \opn\reg{reg} \opn\lreg{lreg} \opn\ini{in} \opn\lpd{lpd}
 \opn\size{size} \opn\sdepth{sdepth}
 \opn\link{link}\opn\fdepth{fdepth}\opn\lex{lex}
 \opn\astab{astab}\opn\dist{dist}
 \opn\dstab{dstab}
 \opn\div{div} \opn\Div{Div} \opn\cl{cl} \opn\Cl{Cl}
 \opn\Spec{Spec} \opn\Supp{Supp} \opn\supp{supp} \opn\Sing{Sing}
 \opn\Ass{Ass} \opn\Min{Min}\opn\Mon{Mon}
 \opn\Ann{Ann} \opn\Rad{Rad} \opn\Soc{Soc}
 \opn\Im{Im} \opn\Ker{Ker} \opn\Coker{Coker} \opn\Am{Am}
 \opn\Hom{Hom} \opn\Tor{Tor} \opn\Ext{Ext} \opn\End{End}
 \opn\Aut{Aut} \opn\id{id}
 
 \opn\nat{nat}
 \opn\pff{pf}
 \opn\Pf{Pf} \opn\GL{GL} \opn\SL{SL} \opn\mod{mod} \opn\ord{ord}
 \opn\Gin{Gin} \opn\Hilb{Hilb}\opn\sort{sort}
 \opn\aff{aff} \opn
 \con{conv} \opn\relint{relint} \opn\st{st}
 \opn\lk{lk} \opn\cn{cn} \opn\core{core} \opn\vol{vol}
 \opn\link{link} \opn\star{star}\opn\lex{lex}\opn\set{set}
 \opn\gr{gr}
 \def\Rees{{\mathcal R}}
 \def\pot#1#2{#1[\kern-0.28ex[#2]\kern-0.28ex]}

 \opn\dirlim{\underrightarrow{\lim}}
 \opn\inivlim{\underleftarrow{\lim}}
 \let\dirsum=\oplus
 
 \let\iso=\cong
 \let\Union=\bigcup
 
 \let\Dirsum=\bigoplus
 
 \let\to=\rightarrow
 \let\To=\longrightarrow
 \def\Implies{\ifmmode\Longrightarrow \else
         \unskip${}\Longrightarrow{}$\ignorespaces\fi}
 \def\implies{\ifmmode\Rightarrow \else
         \unskip${}\Rightarrow{}$\ignorespaces\fi}
 \def\iff{\ifmmode\Longleftrightarrow \else
         \unskip${}\Longleftrightarrow{}$\ignorespaces\fi}

 \let\:=\colon
 \newtheorem{Theorem}{Theorem}[section]
 \newtheorem{Lemma}[Theorem]{Lemma}
 \newtheorem{Corollary}[Theorem]{Corollary}
 \newtheorem{Proposition}[Theorem]{Proposition}
 \newtheorem{Remark}[Theorem]{Remark}
 
 \newtheorem{Example}[Theorem]{Example}
 \newtheorem{Examples}[Theorem]{Examples}
 \newtheorem{Definition}[Theorem]{Definition}

 \let\epsilon\varepsilon
 \let\kappa=\varkappa
 \def\qed{\ifhmode\textqed\fi
       \ifmmode\ifinner\quad\qedsymbol\else\dispqed\fi\fi}
 \def\textqed{\unskip\nobreak\penalty50
        \hskip2em\hbox{}\nobreak\hfil\qedsymbol
        \parfillskip=0pt \finalhyphendemerits=0}
 \def\dispqed{\rlap{\qquad\qedsymbol}}
 
 \opn\dis{dis}
 \def\pnt{{\raise0.5mm\hbox{\large\bf.}}}
 
 \opn\Lex{Lex}

\begin{document}

 \title {Persistence and stability  properties of powers of ideals}

 \author{J\"urgen Herzog and Ayesha Asloob Qureshi}

\address{J\"urgen Herzog, Fachbereich Mathematik, Universit\"at Duisburg-Essen, Campus Essen, 45117
Essen, Germany} \email{juergen.herzog@uni-essen.de}

\address{Ayesha Asloob Qureshi, Abdus Salam School of Mathematical Sciences, GC University,
Lahore. 68-B, New Muslim Town, Lahore 54600, Pakistan} \email{ayesqi@gmail.com}

\thanks{This paper was partially written during the visit of the  second  author at Universit\"at Duisburg-Essen, Campus Essen. The second  author wants to thank the Abdus Salam International Centre for Theoretical Physics (ICTP), Trieste, Italy and the Abdus Salam School of Mathematical Sciences, Lahore, Pakistan  for  supporting her.
}
\subjclass{13C13, 13A30, 13F99, 05E40}
\keywords{Associated prime ideals, polymatroidal ideals, analytic spread, depth}

 \address{ }\email{}

 \begin{abstract}
We introduce the concept of strong persistence and show that it implies persistence regarding the associated prime ideals of the powers of an ideal. We also show that strong persistence is equivalent to a condition on power of ideals studied by Ratliff. Furthermore, we give an upper bound for the depth of powers of monomial ideals in terms of their linear relation graph, and apply this to show that the index of depth stability and the index of stability for the associated prime ideals of polymatroidal ideals is bounded by their analytic spread.
 \end{abstract}

 \maketitle

\section*{Introduction}

In this paper we study stability properties of powers of ideals. It is known by Brodmann \cite{Br2} that the set of associated prime ideals $\Ass(I^k)$ of $I^k$ of an ideal $I$  in a Noetherian ring $R$ stabilizes, that is, there exist an integer $k_0$ such that $\Ass(I^k) = \Ass(I^{k+1})$ for all $k \geq k_0$. The smallest such integer $k_0$, denoted $\astab(I)$, is called the index of stability for the associated prime ideals of $I$.  It  has been shown, see \cite[Example p.2]{McA} and \cite{BHH} that given any number $n$ there exists an ideal $I$ in a suitable graded ring $R$ and a prime ideal $P$ of $R$ such that for all $k \leq n$, $P \in \Ass(I^k)$ if $k$ is even and $P \not \in \Ass(I^k)$ if $k$ is odd.  From both examples it follows that $\astab(I) \geq n$, and hence this index may exceed any given number. The natural question arises whether there exists a uniform bound for $\astab(I)$ in terms of $R$.  In fact, we do not know of  any example of a regular local ring which admits an ideal with $\astab(I) \geq \dim R$. Thus we are tempted to conjecture that if $R$ is a regular local ring or a polynomial ring and $I \subset R$ is an ideal (which we assume to be graded if $R$ is a polynomial ring), then $\astab(I) < \dim R$. One of the results (Theorem~\ref{astabdstab})  in this paper is to show this conjecture holds true for all polymatroidal ideals.

An ideal is said to satisfy the persistence  property  if $\Ass(I)\subset \Ass(I^2)\subset \cdots\subset \Ass(I^k)\subset\cdots$. The above quoted examples show that not all ideals satisfy this persistence property. But for interesting classes of monomial ideals  persistence  has been shown. This includes edges ideal of graphs \cite{MMV}, vertex cover ideals of perfect graphs \cite{FHV} as well a polymatroidal  ideals \cite{HRV}. There is no example known of a squarefree monomial ideal which does not satisfy persistence. The challenging question, raised in \cite{MMV}, is whether such squarefree monomial ideals exist. In their paper \cite[Corollary 2.17]{MMV} the authors show that for a squarefree monomial ideal $I$ the  associated  prime ideals of the powers form an ascending chain, that is, satisfy the persistence  property,  if $I^{k+1}:I=I^k$ for all $k\geq 1$. Indeed they show that under  this assumption a socle element of $S/I^k$ is multiplied by a generator of $I$ to a socle element  of $S/I^{k+1}$. We say that an ideal  in a Noetherian ring satisfies the strong persistence property if for all its localizations their  socle elements are multiplied from one power to the next, see Section 1 for the precise definition. It is then obvious that the strong persistence property  implies the persistence property. In Section 1 we analyze this concept of strong persistence, and show in Theorem~\ref{firstcharacterization},  inspired by the above mentioned result \cite{MMV}, that an ideal $I$ in a Noetherian ring satisfies strong persistence if and only if $I^{k+1}:I=I^k$ for all $k$. We would like to refer to this property of the colon ideals as the Ratliff condition of $I$, because it was Ratliff who showed in \cite{R} that this condition is satisfied for any normal ideal,  and that $I^{k+1}:I=I^k$ for all $k\gg 0$,  in general.

In Theorem~\ref{stwo} we show that an ideal $I$ of positive grade satisfies strong  persistence and hence the Ratliff condition  if we  require that the Rees ring $\Rees(I)$ satisfies Serre's condition $S_2$ (which is less than normality for which one would also need the Serre condition $N_1$), and in Corollary~\ref{normalCM} the stable set of associated prime ideals of $I$ is described in the case that $\Rees(I)$ is normal or Cohen--Macaulay. This description in the normal case is due to Ratliff \cite{R}.

In Section 2 we reformulate strong persistence for monomial ideals and show that all polymatroidal ideals satisfy strong persistence. The persistence property for polymatroidal ideals has already been shown in \cite{HRV}.  The Stanley--Reisner ideal of the canonical triangulation of the projective plane is an example of an ideal which satisfies persistence but not strong persistence.

For a graded ideal $I$ in the polynomial ring $S=K[x_1, \ldots, x_n]$, the function $f(k) = \depth S/ I^k$ is called the depth function of $I^k$. This function has been first considered in \cite{HH1} and subsequently for the edge ideals of a graph in \cite{CMS}. There is no example known of a squarefree monomial ideal whose depth function is not non-increasing. It is easily seen that monomial ideal with the property that all its monomial localizations  have  non-increasing depth function satisfy the persistence property. On the other hand, it is an open question whether ideals satisfying the persistence property have non-increasing depth functions.

As shown in \cite[Proposition 1.2]{HRV} the ideals with the property that all their powers have a linear resolution admit non-increasing depth functions. In Proposition~\ref{linear-powers}, it is shown that this class of ideals satisfy the strong persistence property with respect to the graded maximal ideal of $S$. It is known, that the $\depth S/I^k$ is constant for $k \gg 0$, see \cite{Br1} and \cite{HH1}. The smallest integer $k$ for which $\depth S/I^k = \depth S/I^{k+\ell}$ for all $\ell \geq 1$ is called the index of depth stability of $I$ and denoted by $\dstab (I)$. As for $\astab (I)$ we expect that $\dstab (I) < n$ for all graded ideals $I \subset S$. In fact we prove in Theorem~\ref{astabdstab} that $\dstab I < \dim S$ for any polymatroidal ideal $I$. The proof of this theorem is based on results of Section 3 where we considered the linear relation graph of a monomial ideal. The main result obtained in this section is presented in Theorem~\ref{depthbounds} where an upper bound of the depth of the powers of $I$ is given in terms of the linear relation graph of the ideal. Comparing the data of the linear relation graph of an ideal with its analytic spread $\ell(I)$ which is done in Lemma~\ref{needed}, it can be shown that for the polymatroidal ideals one has $\depth S/ I^{\ell(I) -1} = n - \ell (I)$. This together with the fact that the limit depth of polymatroidal ideal $I$ is equal to $n - \ell(I)$ is substantially used in the proof of Theorem~\ref{astabdstab}.

In general the indices $\astab(I)$ and $\dstab(I)$ are unrelated, as shown by examples given in \cite{HRV}. On the other hand on the evidence of all known examples we conjecture that $\astab(I) = \dstab(I)$ for all polymatroidal ideals $I$.

\medskip
The authors would like to thank Shamila Bayati for several useful discussions  regarding the subjects of this paper  while she was visiting  Universit\"at Duisburg-Essen.

\section{The strong persistence property}

Let $R$ be a Noetherian ring, and let $I\subset R$ be a proper ideal. We follow the usual convention and  denote by $V(I)$ the set of prime ideals containing $I$ and by  $\Ass(I)$ the set of associated prime ideals of $R/I$. For all $P \in \Spec(R)$, we denote by $\mm_P$ the maximal ideal of the local ring $R_P$.

\begin{Definition}
\label{strong}
{\em Let $P \in V(I)$. We say that $I$ satisfies the {\em strong persistence property with respect to $P$} if for all $k$ and all  $f\in (I_P^k:\mm_P )\setminus I_P^k$  there exists $g\in I_P$  such that  $fg\not \in I_P^{k+1}$. The ideal $I$ is said to satisfy the {\em strong persistence property} if it satisfies the strong persistence property for all $P \in V(I)$.
}
\end{Definition}

From above definition it is clear that it is enough to check the strong persistence property for $I$ only for those prime ideals $P$ and $k$ for which $P \in \Ass(I^k)$. Note further that $f\in (I_P^k:\mm_P) \setminus I_P^k$  if and only if $I_P^k: f=\mm_P$.

\medskip
According to \cite{Van} and \cite{HRV}, a proper ideal $I\subset R$ is said to satisfies the {\em persistence property} (with respect to associated ideals), if
\[
\Ass(I)\subset \Ass(I^2)\subset \cdots \subset \Ass(I^k)\subset \cdots.
\]
In other words, if  $P\in \Ass(I^k)$ for some $k$, then $P\in \Ass(I^{k+1})$ as well. The prime ideals $P$ with this property are called {\em stable} prime ideals. The set of stable prime ideals will be denoted by $\Ass^\infty(I)$.

A well-known result of Brodmann \cite{Br2} says that $\Ass(I^k)$ stabilizes, that is,  there exists an integer $k_0$ such that $\Ass(I^k)=\Ass(I^{k+1})$ for all $k\geq k_0$. The result of Brodmann implies that $\Ass^\infty (I)$ is a finite set.

Note that if $I$ satisfies the persistence property, then
\begin{eqnarray}\label{assinfinity}
\Union_{k\geq 1}\Ass(I^k)=\Ass^\infty(I).
\end{eqnarray}

On the other hand if (\ref{assinfinity}) holds then $I$ does not necessarily satisfy the persistence property. Indeed examples are known for which $\mm \in \Ass (I^k)$ for some $k$, but $\mm$ is not in $\Ass^\infty(I)$, see \cite{BHH} and \cite{MMV}.

It is obvious that the strong persistence property  implies the persistence property. In fact, already the following slightly weaker condition implies the persistence property:
\begin{enumerate}
\item[(1)] For all integers $k\geq 1$  and all $P\in \Ass(I^k)$, there exists  $f\in (I_P^k:\mm_P)\setminus I_P^k$ and $g\in I_P$ such that $fg\not\in I_P^{k+1}$.
\end{enumerate}

\medskip
To show that (1) implies persistence let $P\in \Ass(I^k)$. Then  $\mm_P\in \Ass(I_P^k)$. Therefore (1)  implies that there exist  $f\in (I_P^k:\mm_P)\setminus I_P^k$ and $g\in I_P$ such that $fg\not\in I_P^{k+1}$. Therefore  $I_P^{k+1}:fg\neq R_P$. On the other hand, $\mm_P fg\subset I_P^{k+1}$, so that
$I_P^{k+1}:fg=\mm_P$. This shows that $\mm_P\in \Ass(I_P^{k+1})$, and hence $P\in \Ass(I^{k+1})$.

\medskip

In next result we present a family of ideals in the polynomial ring $S=K[x_1, \ldots, x_n]$ satisfying the strong persistence property with respect to the graded maximal ideal $\mm=(x_1, \ldots, x_n)$.

\begin{Proposition}
\label{linear-powers}
Suppose that $I$ is a graded ideal of $S$. If all powers of $I$ have a linear resolution, then $I$ satisfies the  strong persistence property with respect to $\mm$

\end{Proposition}
\begin{proof}
 Let $f \in (I^k : \mm) \setminus I^k$. The existence of $f$ implies that $\depth (S/ I^k) =0$. Let $I$ be generated in degree $d$. Then, since $I^k$ has a linear resolution it follows that the last module in the minimal graded free resolution of $I^k$ is in the form of $S(-(dk+n-1))^{\beta_{n,dk+n-1}}$. It follows that $\Tor_n^S (S/ \mm, S/ I^k)$ is isomorphic to the graded $K$-vector space $S(-(dk+n-1))^{\beta_{n,dk+n-1}}$. On the other hand, there are the following isomorphisms of graded $K$-vector spaces
\[
\Tor_n^S (S/ \mm, S/I^k) \iso H_n (x_1, \ldots, x_n; S/ I^k) \iso ((I^k : \mm) / I^k) (-n)
\]

It follows that $f= f_1 + f_2$ where $f_1 \in I^k : \mm$ is a non-zero homogenous element of degree $dk-1$ and $f_2 \in I^k$. Let $g \in I$ be an element of degree $d$. Then $gf \notin I^{k+1}$ because degree of $\deg gf = d(k+1)-1$.
\end{proof}

\medskip

The following algebraic condition on $I$ characterizes strong persistence.

\begin{Theorem}
\label{firstcharacterization}
The ideal $I\subset R$ satisfies the strong persistence property if and only if $I^{k+1}:I=I^k$ for all $k$.
\end{Theorem}

\begin{proof}
Let $I^{k+1}\:I=I^k$ for all $k\geq 1$. Then $I_P^{k+1}\:I_P=I_P^k$ for all $k\geq 1$.  Thus, replacing  $R$ by $R_P$ and $I$ by $I_P$,  it is enough to show that  $I$ satisfies the strong persistence property with respect to  $\mm$. So suppose that $\mm\in \Ass(I^k)$ for some $k$ and $f\in(I^k\:\mm)\setminus I^k$. By contrary assume that $fg\in I^{k+1}$ for all $g\in I$. Then  $f\in I^k$  because $I^{k+1}\:I=I^k$, a contradiction.

Conversely, assume that  $I$ satisfies the strong persistence property, but $I^{k+1}\:I \neq  I^k$ for some $k\geq 1$. Since the $R$-module $(I^{k+1}\:I)/I^k$ is nonzero, we can choose a minimal prime ideal  $P$ in the support of $(I^{k+1}\:I)/I^k$. Then the $R_P$-module $(I_P^{k+1}\:I_P)/I_P^k$ is nonzero and of finite length. Thus there exists $f\in (I_P^{k+1}:I_P)\setminus I^k_P$ with $\mm_P f\in I^k_P$. Since $I$ satisfies the strong persistence,  there exists $g\in I_P$ such that $fg\not\in I^{k+1}_P$, contradicting the fact that  $f\in I_P^{k+1}:I_P$.
\end{proof}

Ratliff in his paper \cite{R} showed that  $I^{k+1}:I=I^k$ for  $k\gg 0$, and that  $I^{k+1}:I=I^k$  for all $k$ if $I$ is a normal ideal, in other words, if the Rees ring $\mathcal{R}(I)=\Dirsum_{k\geq 0}I^k$ is normal. More recently, Mart\'{i}nez-Bernal, Morey and Villarreal \cite{MMV} showed that for the edge ideal $I$ of any finite simple graph  one has $I^{k+1}:I=I^k$ for all $k$, and they used this fact to show that edge ideals of graphs satisfy the persistence property. Their result inspired us to formulate Theorem~\ref{firstcharacterization}.

The next result provides a generalization of the above quoted result of Ratliff regarding  normal ideals.

\begin{Theorem}
\label{stwo}
Let $R$ be a Noetherian  ring and $I\subset R$  a proper ideal  of $R$ with $\grade(I)>0$.  Suppose  that $R_P/\mm_P$ is an infinite field for all $P\in V(I)$, and that $\mathcal{R}(I)$ satisfies Serre's condition $S_2$. Then $I$ satisfies the strong persistence property. In particular, $I^{k+1}:I=I^k$ for all $k$.
\end{Theorem}

\begin{proof}
Let $P\in V(I)$ and set $Q=P\dirsum \Dirsum_{k\geq 1}I^k$. Then $\Rees(I)/Q\iso R/P$, and hence $Q\in \Spec(\Rees(I))$. Moreover, $\Rees(I)_Q\iso \Rees(I_P)$. It follows that $\dim \Rees(I)_Q=\dim R_P+1\geq 2$, since $\grade(I)>0$. The $S_2$-condition then yields that $\depth \Rees(I_P)\geq 2$. Thus after localization we may  assume that $R$ is local with maximal ideal $\mm$ and that $\depth \Rees(I)\geq 2$, and we have to show that $I$ satisfies the strong persistence property with respect to $\mm$. (The assumptions regarding $R$ and $I$ are preserved upon localization.)
We set $\mathcal{N}(I)=  \Dirsum_{k\geq 0}(I^k:\mm)/I^k$. Obviously,  $\mathcal{N}(I)$ is a graded $\bar{\Rees}(I)=\Rees(I)/\mm\Rees(I)$-module. We may identify $\mathcal{N}(I)$ with the graded  Koszul homology $\bar{\Rees}(I)$-module $H_{n-1}(x_1,\ldots,x_n;\Rees(I))$. Here $x_1,\ldots,x_n$  is a minimal system of generators of $\mm$. The  $k$th graded component  of $H_{n-1}(x_1,\ldots,x_n;\Rees(I))$ is given as
\begin{eqnarray*}
H_{n-1}(x_1,\ldots,x_n;\Rees(I))_k&=&H_{n-1}(x_1,\ldots,x_n;I^k)\iso H_{n}(x_1,\ldots,x_n;R/I^k)\\
&\iso& (I^k:\mm)/I^k=\mathcal{N}(I)_k.
\end{eqnarray*}

It follows that  $\mathcal{N}(I)$ is a finitely generated  $\bar{\Rees}(I)$-module.

We set $\mathcal{M}(I)=\Rees(I):\mm=\Dirsum_{k\geq 0}I^k:\mm$, and obtain the following exact sequence
\[
0\To \Rees(I)\To \mathcal{M}(I)\To \mathcal{N}(I)\To 0
\]
of $\Rees(I)$-modules. Since $\mathcal{N} (I)$ is finitely generated $\Rees(I)$-module, it follows from this exact sequence that  $\mathcal{M}(I)$ is a finitely generated $\mathcal{R}(I)$-module as well.

Since $\depth R>0$, there exists a non-zerodivisor $f\in \mm$ of $R$.  Obviously, $f$ is also a non-zerodivisor on $\mathcal{M}(I))$, so that   $\depth \mathcal{M}(I)>0$. Hence since $\depth \Rees(I)\geq 2$, we conclude that $\depth \mathcal{N}(I)>0$. Now we use the fact that  $\bar{\Rees}(I)$ is a standard graded $K$-algebra, where $K$ is the residue class field of $R$,  and that $K$ is an infinite field. Therefore there exists a homogeneous element $g+\mm I$ of  degree $1$ in $\bar{\Rees}(I)$ with $g\in I$,  which is regular on $\mathcal{N}(I)$. Thus the multiplication map $(I^k\:\mm)/I^k\to  (I^{k+1}\:\mm)/I^{k+1}$ induced by $g$ is an injective map for all $k\geq 0$. As a result, if $f\in (I^k\:\mm)\setminus I^k$ for some $k$, then $gf\not \in I^{k+1}$, as desired.
\end{proof}

\begin{Remark}
\label{more}{\em
Under the assumptions of Theorem~\ref{stwo} one obtains more than just the strong persistence. Indeed, the proof shows that there exists $g\in I\setminus \mm I$ such that for all integers  $k, \ell \geq 0$  and all $f\in (I^k:\mm)\setminus I^k$ we have  $fg^\ell\not  \in I^{k+\ell}$. The proof also shows that under the assumptions of the theorem $\dim_K((I^k:\mm)/I^k)\leq \dim_K((I^{k+1}:\mm)/I^{k+1})$ for all $k$. }
\end{Remark}

An immediate consequence of Theorem~\ref{stwo} is the following result whose statement concerning normality  is due to Ratliff \cite{R}. We denote by $\ell(J)$ the analytic spread of an ideal, that is, the Krull dimension of $\Rees(J)/\mm \Rees(J)$.

\begin{Corollary}\label{normalCM}
Let $R$ and $I$ be as in Theorem~\ref{stwo}. Suppose  that $\mathcal{R}(I)$ is normal or  Cohen--Macaulay. Then $I$ satisfies the strong persistence property. Moreover, $P\in \Ass^\infty(I)$ if and only if $\ell(I_P)=\dim R_P$.
\end{Corollary}

\begin{proof} The statement about strong persistence follows from Theorem~\ref{stwo}. Let $P\in V(I)$. Replacing $I$ by $I_P$ and $R$ by $R_P$ we may assume that $P=\mm$ and have to show that $\mm\in \Ass^\infty(I)$ if and only if   $\ell(I)=\dim R$.

We observe that $\mm\in \Ass^\infty(I)$ if and only if $\Rees(I):\mm\neq \Rees(I)$. Note that $\Rees(I):\mm=\Rees(I):\mm\Rees(I)$. Since $\Rees(I)$ is Cohen--Macaulay, we have $\grade \mm\Rees(I)=\height \mm\Rees(I)$, and $\height \mm\Rees(I)=\dim \Rees(I) -\dim \bar{\Rees}(I)\geq \dim \Rees(I)-\dim \gr_I(R)=1$. On the other hand, the exact sequence
\[
0\to \mm\Rees(I)\to \Rees(I)\to \bar{\Rees}(I)\to 0
\]
induces the exact sequence
\[
0\to  \Rees(I) \to  \Rees(I):\mm\Rees(I)\to \Ext^1(\bar{\Rees}(I),\Rees(I))\to 0.
\]
This shows that $\Rees(I):\mm\neq \Rees(I)$ if and only if $\grade \mm\Rees(I)=1$ which is equivalent to saying that  $\dim \bar{\Rees}(I)=\dim \Rees(I)-1$. Since $\dim R=\dim \Rees(I)-1$, the assertion follows.
\end{proof}

\section{Strong persistence for monomial ideals}

In this section we discuss the concepts of the previous section for monomial ideals. Let $I \subset S= K[x_1, \ldots, x_n]$ be a monomial ideal. Recall that the associated prime ideals of $I$ are all monomial prime ideals, that is, prime ideals which are generated by variables, see for example, \cite{BHbook}, \cite{HHBook}. We denote by $V^*(I)$ the set of monomial prime ideals containing $I$. Thus $\Ass(I) \subset V^*(I)$.

Let $P=(x_{i_1},\ldots, x_{i_r})$ be a monomial prime ideal. The monomial localization of $I$ with respect to $P$, denoted by $I(P)$, is the ideal in the polynomial ring $S(P)=K[x_{i_1},\ldots,x_{i_r}]$ which is obtained from $I$ by applying the $K$-algebra homomorphism $S\rightarrow S(P)$ with $x_j\mapsto 1$ for all $x_j\not\in\{x_{i_1},\ldots,x_{i_r}\}$. For the further discussion, the following fact, shown in \cite[Lemma 1.3]{HRV} and \cite[Lemma 2.11]{FHV}, is of crucial importance: $P \in \Ass(I)$ if and only if $\depth S(P) / I(P) =0$, and this is the case if and only if $\mm_P \in \Ass(I(P))$, where $\mm_P$ is the graded maximal ideal of $S(P)$. It follows that $P \in \Ass^{\infty} (I)$ if and only if $\mm_P \in \Ass^{\infty}(I(P))$.

\medskip
 Since $\bigcup_{k \geq 1} \Ass(I^k) \subset V^*(I)$ and since for all $k$ and all $P \in V^*(I)$ the $K$-vector space $(I(P)^k : \mm_p )/ I(P)^k$ is generated by monomials, one obtains:

\begin{Lemma}
Let $I$ be a monomial ideal. Then $I$ satisfies the strong persistence property if and only if for all $P \in V^*(I)$ and $k$, and all monomials $u \in (I(P)^k:\mm_P) \setminus I(P)^k$  there exists a monomial $v\in I(P)$  such that  $uv \not \in I(P)^{k+1}$.
\end{Lemma}

Let $I$ be the Stanley-Reisner ideal that corresponds to the natural triangulation of the projective plane. Then
\[
I = (x_1x_2x_3, x_1x_2x_4, x_1x_3x_5, x_1x_4x_6, x_1x_5x_6, x_2x_3x_6, x_2x_4x_5, x_2x_5x_6, x_3x_4x_5,
x_3x_4x_6).
\]
It was observed in \cite[Example 2.18]{MMV} that $I$ satisfies the persistence property, while on the other hand $I^3:I\neq I^2$, as can be easily checked. Hence by Theorem~\ref{firstcharacterization} the ideal $I$ does not satisfy the strong persistence theorem. Indeed, $(I^2:\mm)/I^2$ is generated by the residue class of $u=x_1x_2x_3x_4x_5x_6$, and $uI^2\subset I^3$.

\medskip
As mentioned in the previous section it is known by \cite{MMV} that all edge ideals $I(G)$ satisfy the Ratliff condition, namely that $I(G)^{k+1}:I(G)=I(G)^k$ for all $k$, and hence they all satisfy the strong persistence property. Next we present another large class of monomial ideal satisfying the strong persistence property. To this end we first recall the notion of polymatroidal ideals. For the algebraic theory of polymatroids we refer to \cite{HHBook} and \cite{HH2}.

\medskip
The set of bases of a polymatroid of rank $d$ based on $[n]$ is a set $\mathcal{B}\subset \ZZ^n$ of integer vectors $\ab=(a(1),\ldots,a(n))$ with non-negative entries satisfying the following conditions:
\begin{enumerate}
\item[(i)] $|\ab|=\sum_{i=1}^na(i)=d$ for all $\ab\in\mathcal{B}$;
\item[(ii)] (Exchange property) For all $\ab,\bb\in \mathcal{B}$ for which  $a(i)>b(i)$ for some $i$, there exists $j\in [n]$ such that $b(j)>a(j)$ and $\ab-\epsilon_i+\epsilon_j\in \mathcal{B}$. Here $\epsilon_i$ denotes the canonical $i$th unit vector.
\end{enumerate}

If the bases of the polymatroid are all $(0,1)$-vectors, then they form the bases of a matroid. In the following example we give some interesting classes of matroids and polymatroids.

\medskip\noindent
\begin{Examples}\label{examplepolymatroid}{\em

(1) {\em Graphic Matroids:} Let $G$ be a graph with edge set $E(G)=[n]$. The set $\mathcal{B} \subset \ZZ^n$ of bases of the graphic matroid $G$ consists of all vectors $\ab_T = \sum_{i \in E(T)} \epsilon_i$, where $T$ is a spanning forest of $G$.

(2) {\em Transversal polymatroids and matroids:} Given a collection $\mathcal{A} = \{A_1, \ldots, A_r\}$ of subsets of $[n]$, the set of bases $\mathcal{B}$ of the transversal polymatroid attached to $\mathcal{A}$ is

\[
\mathcal{B} = \{\epsilon_{i_1} + \ldots + \epsilon_{i_d} \; : \; i_k \in A_k, 1 \leq k \leq d  \}
\]

If in addition  $i_{p} \neq i_{q}$ for all $1\leq p,q \leq d$ with $p \neq q$, then $\mathcal{B}$ is the set of bases of a transversal matroid.

(3) {\em  Polymatroids of Veronese type:} Given an integer $d$  and a vector $\cb=(c(1), \ldots, c(n))$ with $c(i) \geq 0$. Then the bases of the polymatroid of Veronese type $(d,c)$ is given by all integers vector $\ab=(a(1), \ldots, a(n))$ with $|\ab| =d$ such that $0 \leq a(i) \leq c(i)$ .
}
\end{Examples}

\begin{Definition}
{\em A monomial ideal $I\subset S=K[x_1,\ldots,x_n]$ is called a {\em polymatroidal ideal}, if there exists a set of bases $\B\subset \ZZ^n$ of a polymatroid, such that
$$G(I)=\{\xb^\ab\:\; \ab\in\B\}.$$
Here we denote by $G(I)$ the unique minimal set of generators of a monomial ideal $I$.}
\end{Definition}

Thus a polymatroidal ideal is a monomial ideal, generated in a single degree, satisfying the following condition: for all monomials $u,v\in G(I)$ with $\deg_{x_i}(u)>\deg_{x_i}(v)$, there exists $j\in[n]$ such that $\deg_{x_j}(v)>\deg_{x_j}(u)$ and $x_i(u/x_j)\in G(I)$. Here $\deg_{x_k} (w) = a_k$ if $w = \prod _{k=1}^n x_k^{a_k}$.

\medskip
\noindent
\begin{Proposition}
\label{strongpolymatroidal}
Let $I$ be a polymatroidal ideal. Then $I$ satisfies the strong persistence property.
\end{Proposition}

\begin{proof}
It is known by  \cite[Corollary 2.2]{HRV} that $I(P)$ is again a polymatroidal ideal for all $P\in V^*(I)$. Since powers of polymatroidal ideals are again polymatroidal, see \cite[Theorem 12.6.3]{HHBook}, and since by \cite[Theorem 12.6.2]{HHBook} polymatroidal ideals have linear resolutions, we conclude that all powers of $I(P)$ have a linear resolution. Thus the assertion follows from Proposition~\ref{linear-powers}.
\end{proof}

For the proof of this proposition it was essential to know that not only $I$ but also all monomial localizations of $I$ have a linear resolution. It is conjectured in \cite{BH} that the only monomial ideals with this property are exactly the polymatroidal ideals.

\section{The linear relation graph of a monomial ideal}

In this section we introduce the linear relation graph $\Gamma$ of a monomial ideal $I\subset S=K[x_1,\ldots,x_n]$. This graph will allow us to give upper bounds of the depth of the  powers $I^k$ for $k$ less than a certain number which is determined by $\Gamma$.

\begin{Definition}
{\em Let $G(I)=\{u_1,\ldots,u_m\}$. The {\em linear relation graph} $\Gamma$ of $I$ is the graph with edge set
\[
E(\Gamma)=\{\{i,j\}\: \text{there exist $u_k,u_l\in G(I)$ such that $x_iu_k=x_ju_l$}\}
\]
and vertex set $V(\Gamma)=\Union_{\{i,j\}\in E(\Gamma)}\{i,j\}$.
}
\end{Definition}

\begin{Example}
\label{edgeideals}
{\em
Let $I_G$ be the edge ideal of the finite simple graph $G$ on the vertex set $[n]$. Then the linear relation graph $\Gamma$ of $I_G$ has edge set
\[
\{ \{i,j\} \; : \text{$i, j \in V(G)$ and $i$ and $j$ have a common neighbor in $G$}\}.
\]
In particular, let $G$ be  an odd cycle with edges $\{i, i+1\}$ for $i=1, \ldots n$. For simplicity of notation, here and in the following,  any  $i$ exceeding $n$ stands the remainder of $i$ modulo $n$. With the notation introduced,  $\Gamma$ is an odd cycle with  edges $\{i,i+2\}$ for $i=1, \ldots, n$.

On the other hand, if $G$ is an even cycle, then $\Gamma$ is the graph with two connected components $\Gamma_1$ and $\Gamma_2$  where $\Gamma_1$ is a cycle with $$E(\Gamma_1) =\{\{2i,2i+2\}\; : i= 1, \ldots, n/2 \}$$ and $\Gamma_2$ is a cycle with $$E(\Gamma_2) =\{\{2i-1,2i+2\}\; : i= 1, \ldots, n/2 \}.$$

}
\end{Example}

\begin{Theorem}
\label{depthbounds}
Let $I\subset S=K[x_1,\ldots,x_n]$ be a monomial ideal generated in a single degree whose linear relation graph $\Gamma$ has $r$ vertices and $s$ connected components. Then
\[
\depth S/I^t\leq n-t-1 \quad \text{for $t=1,\ldots, r-s$}.
\]
\end{Theorem}

\begin{proof}
It is enough to show that $H_t(x_1, \ldots, x_n; I^t) \neq 0$ for $t=1, \ldots, r-s$. Let $T \subset \Gamma$ be a spanning forest of $\Gamma$, i.e, a subgraph of $\Gamma$ which is a forest and for which $V(T) = V(\Gamma)$. This forest has $r-s$ (distinct) edges, say, $$\{i_1,j_1\}, \{i_2,j_2\},\ldots,\{i_{r-s},j_{r-s}\}.$$ For a suitable labeling of the edges we may assume that for all $k$, $j_k$ is a free vertex of the forest with edges $\{i_1,j_1\}, \{i_2,j_2\},\ldots,\{i_{k},j_{k}\}$. In particular, we have $j_k\not\in\{i_1,\ldots,i_k,j_1,\ldots,j_{k-1}\}$.

By the definition of $\Gamma$, it follows that to each edge $\{i_k,j_k\}$ belongs to a cycle $z_k=u_{p_k}e_{j_{k}}-u_{q_k} e_{i_k}$ in $K(x_1, \ldots, x_n ; I)$ where $u_{p_k}$ and $u_{q_k}$ are suitable elements in $G(I)$. Then $z=z_1\wedge z_2\ldots \wedge z_{t}$ is a non-trivial cycle in $K_{t}(x_1,\ldots,x_n;I^{t})$. Indeed, $z\neq 0$, because in $z$ the basis element $e_{j_1}\wedge e_{j_2}\wedge \ldots \wedge e_{j_{t}}$ appears in the expansion of the wedge product only once (with coefficient $u_{p_1}, u_{p_2}\cdots u_{p_{t}}$). The cycle $z$ cannot be a boundary of $K_{t}(x_1,\ldots,x_n;I^{t})$ because its coefficients all belong to $I^{t}\setminus \mm I^{t}$, since $I$ is generated in single degree. This shows that $[z]\neq 0$ in $H_{t}(x_1,\ldots,x_n;I^{t})$, and proves the theorem.
\end{proof}

Applying Theorem~\ref{depthbounds} to  edge ideals of cycles,  it follows from the results in Example~\ref{edgeideals} that $\depth S/I_G^t\leq n-t-1$ for $t=1,\ldots, k$, where $k=n-1$ if $G$ is an odd cycle of length $n$, and $k=n-2$, if $G$ is an even cycle of length $n$.

For example, if $n=5$, one has  $\depth S/I_G=\depth S/I_G^2=2$ and $\depth S/I_G^3=\depth S/I_G^4=0$. This shows that in general the upper bound for the depth of the powers given in Theorem~\ref{depthbounds} is not strict.

\medskip
In a particular situation of Theorem~\ref{depthbounds}, we obtain the following additional information.

\begin{Corollary}\label{additional}
Let $I\subset S=K[x_1,\ldots,x_n]$ be a monomial ideal generated in single degree with linear relation graph $\Gamma$. Assume that $|V(\Gamma)| = n$ and that $\Gamma$ is connected.  Let $T \subset \Gamma$ be a spanning tree of $\Gamma$ with edges $\{i_1,j_1\}, \{i_2,j_2\},\ldots,\{i_{n-1},j_{n-1}\}$ such that for all $k$, $j_k$ is a free vertex of the tree with edges $\{i_1,j_1\}, \{i_2,j_2\},\ldots,\{i_{k},j_{k}\}$. For $k=1, \ldots, n-1$, let $z_k=u_{p_k}e_{j_{k+1}}-u_{q_k} e_{i_k}$ be the cycle corresponding to the edge $\{i_k,j_k\}$ and let $r$ be the unique element in $[n] \setminus \{j_1, \ldots, j_{n-1}\}$. Then $\depth(S/I^{n-1}) = 0 $ and $u_{p_1}\ldots u_{p_{n-1}}/x_r$ is a socle element of $S/I^{n-1}$. Moreover, $\mm \in \Ass(I^k)$ for all $k \geq n-1$.
\end{Corollary}

\begin{proof}
The assumptions imply that $z=u_{p_1}\ldots u_{p_{n-1}}e_{j_1}\wedge e_{j_2}\wedge \ldots \wedge e_{j_{t}}+\cdots$ is a cycle in $K_{n-1}(x_1,\ldots,x_n;I^{n-1})$ whose homology class $[z]\in H_{n-1}(x_1,\ldots,x_n;I^{n-1})$ is non-trivial and for which $e_{j_1}\wedge e_{j_2}\wedge \ldots \wedge e_{j_{t}}$ appears only once in the expansion of $z$, see the proof of Theorem~\ref{depthbounds}. Now we use the fact that
\[
H_n(x_1,\ldots,x_n;S/I^{n-1})\iso H_{n-1}(x_1,\ldots,x_n;I^{n-1}).
\]
This isomorphism is established as follows: a generator  $[(u+I^{n-1})e_1\wedge e_2\wedge \cdots \wedge e_n]\in H_n(x_1,\ldots,x_n;S/I^{n-1})$ with $u\in I^{n-1}:\mm$ is mapped to $[\sum_{i=1}(-1)^{i+1}ux_ie_1\wedge \cdots e_{i-1}\wedge e_{i+1}\wedge \cdots \wedge e_n]\in  H_{n-1}(x_1,\ldots,x_n;I^{n-1})$.

It follows that $[(u+I^{n-1})e_1\wedge e_2\wedge \cdots \wedge e_n]$ with $u=(-1)^{r+1}u_{p_1}\ldots u_{p_{n-1}}/x_r$ is mapped to our given non-trivial homology class $[z]$. This implies that $w=u_{p_1}\ldots u_{p_{n-1}}/x_r\in (I^{n-1}:\mm)\setminus I^{n-1}$, as desired.

Say, $I$ is generated in degree $d$. Then $I^{k}$ is generated in degree $dk$ for all $k$ and  $\deg w=d(n-1)-1$. Let $v$ be an arbitrary generator of $I$ (of degree $d$). Then $v^\ell w\in I^{n-1+\ell}:\mm$  for all $\ell\geq 0$ and $\deg v^\ell w=d(n-1+\ell)-1$. This shows that $v^\ell w\in (I^{n-1+\ell}:\mm)\setminus
I^{n-1+\ell}$ and implies that $\mm \in \Ass(I^k)$ for all $k \geq n-1$.
\end{proof}

We come back to Example~\ref{edgeideals} and compute a socle element of $S/I_G^{n-1}$ where $n$ is odd. $I_G=(u_1,\ldots,u_n)$ with $n$ odd, and $u_i=x_ix_{i+1}$ for $i=1,\ldots,n$. Then $\Gamma$ is a cycle on the vertex set $[n]$  and the subgraph $T\subset \Gamma$ with edges $\{2i-1,2i+1\}$ for $i=1,\ldots,n-1$ is a spanning forest of $\Gamma$. For $i=1,\ldots,n-1$ we obtain the $1$-cycles $z_i= u_{2i-1}e_{2i+1}-u_{2i+1}e_{2i-1}$. Since $1$ is different from all the numbers $2i+1$ with $i=1,\ldots,n-1$, it follows from Corollary~\ref{additional} that
\[
u=(\prod_{i=1}^{n-1}u_{2i-1})/x_1= (\prod_{i=1}^{n}u_i)/u_{n-1}x_1=x_1x_{n-1}x_n \prod_{i=2}^{n-2}x_i^2
\]
is a socle element of $S/I_G^{n-1}$.

Actually, as shown in the paper by Chen et al. \cite[Lemma 3.1]{CMS}, it is shown that for an odd cycle $G$ of length $2k+1$, one has already that $\mm \in \Ass I^k$ and that $x_1 x_2 \ldots x_{2k+1}$ is a socle element of $S/ I^k$. This example shows that our bound for the depth stability is in general not a strict bound.

\section{Stability indices for polymatroidal ideals}

Let $R$ be Noetherian ring and $I\subset R$ a proper ideal. Following \cite{HRV} and \cite{MMV} we define {\em the index of stability for the associated prime ideals of the powers  of $I$}   to be  the number
\[
\astab(I)=\min\{k\:\;  \Ass(I^{\ell})=\Ass(I^k) \text{ for  all $\ell\geq k$}\}.
\]
Furthermore,  we define the {\em index of depth stability}  of $I$, as introduced in \cite{HRV}, to be the number
\[
\dstab(I)=\min\{k\:\; \depth S/I^\ell=\depth S/I^k \text{ for  all $\ell\geq k$}\}.
\]

The main objective of this section is the proof of the following

\begin{Theorem}
\label{astabdstab}
Let $I\subset S=K[x_1,\ldots,x_n]$ be a polymatroidal ideal. Then $$\astab(I),\dstab(I)< \ell(I).$$
In particular, $\astab(I),\dstab(I)<n$.
\end{Theorem}

\medskip
We shall need the following

\begin{Lemma}
\label{needed}
Let $I$ be a monomial ideal and $\Gamma$ be the linear relation graph of $I$. Suppose $\Gamma$ has $r$ vertices and $s$ connected components. Then
\[
\ell(I) \geq r-s+1,
\]
and equality holds if $I$ is a polymatroidal ideal.
\end{Lemma}

\begin{proof}
 Let $G(I)=\{\xb^{\ab_1},\ldots, \xb^{\ab_m}\}$ and $\B=\{\ab_1,\ldots,\ab_m\}$ be the set of exponent vectors of $G(I)$,  and let $A$ be the $m\times n$-matrix whose row vectors are $\ab_1,\ldots,\ab_m$. Then $\ell (I)$ is the rank of $A$.

Let $U\subset \QQ^n$ be the $\QQ$-vector space spanned  by all vectors $\ab_k-\ab_\ell$ with $\ab_k,\ab_\ell \in \B$ and such that $\ab_k-\ab_\ell=\pm \epsilon_{ij}$ for some $i<j$. Furthermore, let $\Gamma_1,\ldots, \Gamma_s$ be the connected components of $\Gamma(I)$. Then $U=U_1\dirsum U_2\dirsum \cdots \dirsum U_s$, where $U_t$ is $\QQ$-subspace of $U$ spanned by all vectors $\ab_k-\ab_\ell$ with $\ab_k-\ab_\ell=\pm \epsilon_{ij}$ and $\{i,j\}\in E(\Gamma_t)$.

We claim that $\dim U_t=|V(\Gamma_t)|-1$. Indeed, let $i$ be an edge of $\Gamma_t$. Then $\epsilon_i\not\in U_t$ since for each vector of $U_t$, the sum of its components is zero. The desired formula for $\dim U_t$ follows therefore from the identity $U_t+\QQ\epsilon_i=\Dirsum_{j \in V(\Gamma_t)}\QQ\epsilon_j$. To see that this identity holds, let $j\in V(\Gamma_t)$. Since $\Gamma_t$ is connected, there exist vertices $i=i_0,i_1,\ldots,i_k=j$ in $\Gamma_t$ such that $\{i_l,i_{l+1}\}$ is an edge of $\Gamma_t$ for all $l$. We prove by induction of the length $k$ of this path connecting $i$ and $j$ that $\epsilon_j\in U_t+\QQ\epsilon_i$. If $k=1$, then $\epsilon_{ij}=\epsilon_i-\epsilon_j\in U_t$, and hence $\epsilon_j\in U_t+\QQ\epsilon_i$. Now let $k>1$. By what we have seen, we have that $\epsilon_1\in U_t+\QQ\epsilon_i$. Since $i_1$ is connected to $j$ by a path of length $k-1$, we may apply our induction hypothesis and conclude that $\epsilon_j\in U_t+\QQ\epsilon_i$.

Summarizing what we found so far, we see that $\dim U=r-s$ where $s$ is the number of connected components of $\Gamma(I)$. Since all vectors in $U$ have component sum equal to zero, it follows that $U$ is a proper subspace of the vector space which is spanned by the column vectors of $A$. It follows that $r-s=\dim U\leq \rank A-1=\ell(I)-1$. 

Now assume that $I$ is polymatroidal. Then $\B$ is the basis of a polymatroid. We will show that in this case $\dim U= \rank A-1$. To see this we first notice that $\ab_k-\ab_\ell\in U$ for all $\ab_k,\ab_\ell\in \B$. We prove this by induction on the distance of $\ab_k$ and $\ab_\ell$ which, according to \cite{HH2}, is defined the be the number $$\dist(\ab_k,\ab_\ell)= 1/2(\sum_{i=1}^n|\ab_k(i)-\ab_\ell(i)|).$$
If $\dist(\ab_k,\ab_\ell)=1$, then $\ab_k-\ab_\ell\in U$ by the definition of $U$. Suppose now that $\dist(\ab_k,\ab_\ell) >1$. Then the exchange property implies that there exist $i$ and $j$ with  $\ab_k(i)>\ab_\ell(i)$ and $\ab_k(j)<\ab_\ell(j)$ such that $\ab:= \ab_k-\epsilon_{ij}\in \B$. Since $\dist(\ab-\ab_\ell)<\dist(\ab_k,\ab_\ell)$ our induction hypothesis implies that $\ab-\ab_\ell\in U$. Hence, since $\ab-\ab_k\in U$, we conclude that $\ab_k-\ab_\ell\in U$ as well.

Now by using that $\ab_k-\ab_\ell\in U$ for all $\ab_k,\ab_\ell\in \B$, we obtain that
\[
U+\QQ\ab_1=\QQ \B.
\]

Hence, since $\ab_1\not\in U$ (because its coefficient sum is not zero), we have $\rank  A =\dim \QQ\B=\dim(U+\QQ \ab_1)=\dim U+1$, as desired.
\end{proof}

\begin{Lemma}\label{ayesha}
Let $I$ be a polymatroidal ideal and $P \in V^*(I)$. Then $\ell (I(P)) \leq \ell (I) $.
\end{Lemma}

\begin{proof}
We know from \cite[Corollary 2.2]{HRV} that $I(P)$ is a polymatroidal ideal. Let $\P$ (resp. $\P'$) be the polymatroid which defines $I$ (resp. $I(P)$). Let $A$ be the $m\times n$-matrix whose row vectors define the bases of $\P$. Since $I(P)$ is obtained from $I$ by the substitution $x_i \mapsto 1$ for $i \notin P$, the matrix $A'$ whose row vector define the bases of $P'$ is obtained from $A$ by removing the columns of $A$ which correspond to exponents of $x_i \notin P$ and removing the rows of $A$ which do not correspond to the minimal generators of $I(P)$. It follows that $\rank A' \leq \rank A$, as desired.
\end{proof}

\begin{proof}[Proof of Theorem~\ref{astabdstab}]
Combining  Theorem~\ref{depthbounds} with Lemma~\ref{needed}  we obtain that  $$\depth S/I^{\ell(I)-1}=n-\ell(I).$$ By \cite[Corollary 2.5]{HRV},   $\depth S/I^{k}=n-\ell(I)$ for all $k\gg 0$ and by \cite[Proposition 2.3]{HRV}, $I$ satisfies the persistence property ($I$ even satisfies the strong persistence property, see Proposition~\ref{strongpolymatroidal}, so that $\depth S/I^{k}\geq n-\ell(I)$ for all $k$. We conclude that $\dstab(I)<\ell(I)$.

In order to prove that $\astab(I)<\ell(I)$, we observe that $P\in \Ass^\infty(I)$ if and only if $\ell(I(P))=\dim S(P)$, see Corollary~\ref{normalCM}. Now Theorem~\ref{depthbounds} together with Lemma~\ref{needed} imply that $\depth S(P)/I(P)^{\ell(I(P))-1}=0$ which is equivalent to saying that $P\in \Ass(I^{\ell(I(P))-1})$. It follows that
$\astab(I) \leq \max\{\ell(I(P))-1\:\; P\in \Ass^\infty(I)\}$. This implies that $\astab(I) <\ell(I)$ since $\ell(I(P))\leq \ell(I)$ for all $P\in \Ass^\infty(I)$, see Lemma~\ref{ayesha}.
\end{proof}

In view of Theorem~\ref{astabdstab} it is of interest to determine the analytic spread of polymatroidal ideals, or equivalently the number $r-s+1$ attached to the linear relation graph of a polymatroidal ideal. We discuss the polymatroidal ideals attached to the polymatroids introduced in Example~\ref{examplepolymatroid}. To begin with we first discuss the graphic matroids. For that purpose we recall some facts from graph theory. Let $G$ be a finite simple graph with vertex $V(G)$ and edge set $E(G)$. We denote by $c(G)$ the number of connected components of $G$, and for a subset $T \subset V(G)$ we denote by $G_T$ the graph restricted to the vertex set $T$. In particular for any vertex $v \in V(G)$ we set $G \setminus v =G_{ V(G) \setminus \{v\}}$.  A vertex $v$ of $G$ is called a {\em cutpoint} if $c(G) < c(G \setminus v)$. A connected graph with no cutpoints is called {\em biconnected}.

\medskip
\noindent
Let $G$ be a graph. The following facts are known from graph theory:

\medskip
(1) A maximal biconnected subgraph of $G$ is called a {\em biconnected component} of $G$. Let $G_1, \ldots, G_t$ be the biconnected components of $G$. Then $G = \bigcup G_i$ and any two distinct biconnected components intersect at most in a cutpoint.

\medskip
(2) Suppose $G$ is a biconnected. Then any two distinct edges belong to a cycle.

\begin{Proposition}\label{graphicmatroids}
Let $G$ be a graph and $I$ be the matroidal ideal attached to the graphic matroid of $G$. Let $G_1, \ldots, G_s$ be the biconnected components of $G$ which  contain more than one edge. Then $\ell(I) = |E(\bigcup_{i=1}^{s} G_i)| - s +1$.
\end{Proposition}

\begin{proof}
Let $\Gamma$ be the linear relation graph of $I$. We see from Lemma~\ref{needed} that it is enough to show that $\Gamma$ has $|E(\bigcup_{i=1}^{r} G_i)|$ vertices and $r$ connected components. The matroidal ideal $I \subset K[x_i | e_i \in E(G) ]$ attached to the graphic matroid of $G$ is generated by the monomials $u_F = \prod_{e_i \in E(F)} x_i$, where $F$ is a spanning forest of $G$. Let $m$ be the number of biconnected components of $G$. First observe that each spanning forest $F$ of $G$ can be written as $\bigcup_{i=1}^m T_i$ where $T_1, \ldots, T_m$ are spanning trees of the distinct biconnected components of $G$. From the definition of $\Gamma$ we see that $\{i,j\}$ is an edge of $\Gamma$ if there exists a spanning forest $F$ with $e_i \in E(F)$ such that $(E(F) \setminus \{e_i\}) \cup \{e_j\}$ is also a spanning forest of $G$. If $e_i$ and $e_j$ are edges in different biconnected components of $G$, then for any spanning forest $F$ with $e_i \in F$, the subgraph with edge set $(E(F) \setminus \{e_i\}) \cup \{e_j\}$ contains a cycle. Indeed, this cycle is contained in the biconnected component which contains the edge  $e_j$. Hence if $G_1,\ldots, G_t$, $t\geq s$,  are all the biconnected components of $G$, then  $\Gamma$ is the disjoint union of the linear relation graphs $\Gamma_k$ of $G_k$ for $k=1,\ldots,t$. Obviously, if $G_k$ contains only one edge, then $\Gamma_k=\emptyset$. Thus $\Gamma=\Union_{k=1}^r\Gamma_k$, and this union is disjoint.

Next we show that each $\Gamma_i$  for $i=1,\ldots,s$ is a complete graph, which then yields the desired conclusion. Indeed, if $e_i$ and $e_j$ belong to same biconnected component $G_k$ of $G$,  there is a cycle $C$ in $G$ which passes  through $e_i$ and $e_j$.  We can construct a spanning forest $F$ of $G$ such that $E(C) \setminus \{e_j\}$ is contained in $E(F)$. Then $(E(F) \setminus \{e_i\}) \cup \{e_j\}$ is again a spanning tree of $G_k$. This shows that $\{i,j\} \in E(\Gamma)$.
\end{proof}

\begin{Example}\label{graphicexample}{\em
 Figure~\ref{graphic} shows a graph $G$ with 3 biconnected components and the linear relation graph $\Gamma$ of the matroidal ideal $I$ attached to $G$. From Proposition~\ref{graphicmatroids} we see that $\ell(I) =6$.
\begin{figure}[hbt]
\begin{center}
\psset{unit=1.5cm}
\begin{pspicture}(4.5,-0.5)(4.5,1)
\rput(-3.5,0){
\rput(3.52,-2){\psecurve[linewidth=1pt](1,1.8)(1.65,2.5)(1,3.2)(0.3,2.5)(1,1.8)(1.65,2.5)(1,3.2)(0.3,2.5)}
\rput(4.5,-2){\psecurve[linewidth=1pt](1,2.3)(1.65,2.5)(1,2.7)(0.3,2.5)(1,2.3)(1.65,2.5)(1,2.7)(0.3,2.5)}
\rput(5.5,-2){\psecurve[linewidth=1pt](1,1.8)(1.35,2.5)(1,3.2)(0.3,2.5)(1,1.8)(1.35,2.5)(1,3.2)(0.3,2.5)}
\pspolygon(4.5,0)(5,0.5)(4.5,1)(4,0.5)
\psline(5,0.5)(6,0.5)
\pspolygon(6,0.5)(6.5,0)(6.5,1)
\rput(4.5,0){$\bullet$}
\rput(5,0.5){$\bullet$}
\rput(4.5,1){$\bullet$}
\rput(4,0.5){$\bullet$}
\rput(5,0.5){$\bullet$}
\rput(6,0.5){$\bullet$}
\rput(6.5,0){$\bullet$}
\rput(6.5,1){$\bullet$}
\rput(5.5,-0.5){$G$}
}
\rput(1.5,0){
\pspolygon(4.5,0)(5,0.5)(4.5,1)(4,0.5)
\pspolygon(6,0.5)(6.5,0)(6.5,1)
\psline(4.5,1)(4.5,0)
\psline(4,0.5)(5,0.5)
\rput(4.5,0){$\bullet$}
\rput(5,0.5){$\bullet$}
\rput(4.5,1){$\bullet$}
\rput(4,0.5){$\bullet$}
\rput(5,0.5){$\bullet$}
\rput(6,0.5){$\bullet$}
\rput(6.5,0){$\bullet$}
\rput(6.5,1){$\bullet$}
\rput(5.5,-0.5){$\Gamma$}
}
\end{pspicture}
\end{center}
\caption{}\label{graphic}
\end{figure}
}
\end{Example}

Next we determine the analytic spread of transversal polymatroids. Unfortunately, for transversal matroids we do not know the answer. Let $I$ be a transversal polymatroidal ideal, then $I$ has a unique presentation $I=\prod_{i=1}^{m} P_i$, where each $P_i$ is a monomial prime ideal. Collecting the factors $P_i$ which are principal ideals it follows that $I= u J$ where $u$ is a monomial and $J$ is product of remaining monomial prime ideals $P_i$. Hence $J$ is also a transversal polymatroidal ideal. Since $\ell(I) = \ell(J)$, we may assume that $u=1$, and hence from the very beginning we may assume that none of the $P_i$ is a principal ideal. For $i= 1, \ldots, m$, we set

\[
F_i = \{j | \; x_j \in P_i \}.
\]

As usual we denote by $\langle F_1, \ldots, F_m \rangle$ the simplicial complex $\Delta$ generated by the $F_1,\ldots,F_m$. In other words, $\Delta$ is the simplicial complex on $\Union_{i=1}^m F_i$ with the property that $F\in \Delta$ if and only if $F\subset F_i$ for some $i$.

\begin{Proposition} \label{transversal}
Let $\Gamma$ be the linear relation graph of the transversal polymatroidal ideal $I= \prod_{i=1}^m P_i$, and let $\Delta_1,\ldots,\Delta_s$  be the connected components of the simplicial complex $\Delta=\langle F_1, \ldots, F_m \rangle$. Then $\Gamma$ has $s$ connected components $\Gamma_1,\ldots,\Gamma_s$, and for $k=1,\ldots,s$, the connected component $\Gamma_k$ is  the complete graph on the vertex set of $\Delta_k$.
\end{Proposition}

\begin{proof}
It is enough to show that $\{i,j\}$ is an edge of $\Gamma$ if and only if $i$ and $j$ belong to the vertex set of $\Delta_k$ for some $k$. Let $i$ and $j$ be two vertices of $\Delta_k$. We may assume that $\Delta_k = \langle F_1, \ldots, F_t \rangle $ with $t \leq m$. For $\Delta_k$, we define the so-called intersection graph whose vertices are the $F_i$, and $\{F_i, F_j\}$ is an edge of this graph if $F_i \cap F_j \neq \emptyset$. This graph is connected because $\Delta_k$ is connected. In particular, we may choose a spanning tree from this graph. This spanning tree has $t-1$ distinct edges. Therefore, there exist $t-1$ distinct pairs $(F_p, F_q)$ such that $I_{pq}=F_p \cap F_q \neq \emptyset$ for $p,q \in \{1, \ldots, t\}$. We construct a monomial $u$ of degree $m-1$ whose support consists of $t-1$ variables chosen with indices from each of $I_{pq}$ and the remaining variables chosen with indices from each of $F_l$ with $t+1 \leq l \leq m$. Then $x_i u$ and $x_j u$ belong to the set of generators of $I$ and give the edge $\{i,j\}$ in $\Gamma$, as required.

Conversely, let $\{i,j\}$ be an edge of $\Gamma$. Let $V(\Delta) = \{1, \ldots, n\}$ be the vertex set of $\Delta$. Then $[n]$ is the disjoint union of $V(\Delta_k)$, $k= 1, \ldots, s$. Therefore any monomial has a unique presentation $w = w_1 \ldots w_s$ with $\supp(w_k) \in V(\Delta_k)$. Moreover, if $w \in G(I)$ then $\deg w_k$ is the number of $F_i$'s which belong to $\Delta_k$. Since $\{i,j\}$ is an edge of $\Gamma$, there exist $u,v \in G(I)$ such that $x_i u= x_j v$. Suppose that $i\in V(\Delta_k)$ and $j \in V(\Delta_l)$ with $k \neq l$. Then $x_i u_k = v_k$, which is impossible by degree reasons. This completes the proof.
\end{proof}

Finally we discuss the case of polymatroidal ideals of Veronese type. Let $I$ be such an  ideal. Thus we may  assume that $I$ is generated in degree $d$ and  that the exponent vector of any $w \in G(I)$ is bounded componentwise by the vector $\cb= (c_1, \ldots, c_n)$. To exclude the trivial cases we may assume without loss of generality that $\sum_{i=1}^n c_i > d$ and $c_i >0$. It is easy to see that $\Gamma$ has only one connected component and $V(\Gamma) = \supp(I)$ where $\supp(I)$ denotes the set of integers $i$ with the property that $x_i$ divides one of the generators of $I$. Indeed, let $i, j \in \supp(I)$ with $i\neq j$. We may assume that $i=1$ and $j=2$. We choose a monomial $u \in G(I)$ with $u=x_1^{c_1} x_2^{c_2-1} v$ where $v$ is the monomial of degree $d- c_1 - c_2 +1$ with $1,2 \notin \supp(v)$, and whose exponent vector is componentwise bounded by $(c_3, \ldots, c_n)$. Such a monomial exist since $\sum_{i=3}^n c_i \geq d- c_1-c_2+1$. Therefore the monomial $x_2 u / x_1$ has degree $d$ and has exponent vector bounded by $\cb$. Hence $x_2 u / x_1 \in G(I)$, which implies $\{1,2\}$ is an edge of $\Gamma$. In particular, $\ell (I) = \supp (I)$.

{}

 \end{document}